\documentclass[11pt, oneside]{article}   	
\usepackage{geometry}                	
\title{On $\ell$-th roots and division by $\ell$ }

\author{Josep M. Miret, Jordi Pujol\`as, Nicolas Th\'eriault\\ \\Escola Polit\`ecnica Superior, Universitat de Lleida, \\Jaume II 69, 25001 Lleida, Spain, \\Departamento de Matem\'atica y Ciencia de la Computaci\'on, \\Universidad de Santiago de Chile, \\Avda. Libertador Bernardo O'Higgins 3363, Santiago de Chile, Chile, \\ \\josepmaria.miret@udl.cat, jordi.pujolas@udl.cat, nicolas.theriault@usach.cl}

\usepackage{amsmath,amssymb,amsthm}

\newtheorem{theorem}{Theorem}[section]
\newtheorem{lemma}[theorem]{Lemma}
\newtheorem{proposition}[theorem]{Proposition}
\newtheorem{corollary}[theorem]{Corollary}

\theoremstyle{definition}

\newtheorem{definition}[theorem]{Definition}
\newtheorem{example}[theorem]{Example}

\theoremstyle{remark}
\newtheorem{remark}[theorem]{Remark}

\newcommand\divi{\mbox{\rm div}}

\begin{document}
\maketitle

\abstract{We give a characterization of the codomain $[\ell]E(k)$ of the multiplication-by-$\ell$ map $[\ell]$ in the case of elliptic curves over a field $k$ of characteristic $\ne 2,3$ with $\ell$-torsion $E[\ell]=\langle W_1,W_2 \rangle$ fully defined over $k$, for primes $\ell$ different from the characteristic. We show that a point $Q\in E(k)$ lies in $[\ell]E(k)$ if and only if $h_{W_1}(-Q)$ and $h_{W_2}(-Q)$ are $\ell$-powers of $k$, where 
$h_{W_1}$ and $h_{W_2}$ are functions on $E$ with divisor $\divi(h_{W_i})=\ell W_i- \ell P_{\infty}$. 
Our characterization leads to an effective procedure to find pre-images of $[\ell]$ by solving an order $\ell$ system of linear equations and computing a polynomial gcd.}

\section{Introduction}
Let $E$ be an elliptic curve in the Weierstrass model over a field $k$ of characteristic different from $2$ or $3$, and let $\ell$ be a prime different from the characteristic of $k$. Our purpose is to compute the pre-image set of the multiplication-by-$\ell$ map $[\ell]$ by means of two $\ell$-th roots and some basic algebraic operations. The solution to this problem for $\ell=2$ is well known and appears in reference books \cite{Hu,ST}.  The case $\ell=3$ appeared recently in \cite{PT}. In this paper we compute the pre-images for any $\ell$.

Throughout, we suppose the $\ell$-torsion has a basis $E[\ell]=\langle W_1, W_2\rangle$ such that $W_1,W_2\in E(k)$. 
With these assumptions we solve the following computational problem: 
given $Q\in E(k)$, to decide whether or not $Q$ lies in the image of the group of rational points $E(k)$ by the multiplication-by-$\ell$ map $[\ell]$ in $E$: $$Q\overset{?}{\in} [\ell]E(k).$$
In case it does, we compute the set of points $P\in E(k)$ such that $\ell P=Q$.

The outline of our solution is the following. We first consider a Miller function $h_{W_i}$ for each point of our basis $\{ W_1, W_2\}$.  We prove $Q\in [\ell]E$ if and only if  $h_i(-Q)$ are $\ell$-th powers in $k$ for $i=1,2$ (see Theorem \ref{main}). 
As a consequence, we associate a unique (non canonical) pair of independent $\ell$-th roots 
\begin{equation*}
\omega_1=\sqrt[\ell]{h_1(-Q)},\omega_2=\sqrt[\ell]{h_2(-Q)}\in k
\end{equation*}
to each of our $P$'s. 
We then define a $\ell\times \ell$ linear equation, the solution of which provides with a univariate polynomial $\varphi(x)$ with $\omega_1,\omega_2$ in the coefficients and $x_P$ as a root. Finally, we compute $x_P$ as the greatest common divisor of $\varphi(x)$ (obtained from $\omega_1,\omega_2$) and a well known polynomial $\phi_{\ell,Q}(x)$ obtained from $Q$ and the $n$-division polynomials $\psi_{n}(x)$.

We know of a solution to this problem by McMurdy~\cite{Mc}. With a different approach, he 
 obtains a system of equations, a solution of which 
 can only be obtained by factoring a degree $\ell^2$ polynomial. Along the lines of \cite{MMRV}, we need to factor two degree $\ell$ polynomials instead, which is more efficient.

\section{Miller functions $h_{W_i}$}
For our convenience, we write the $\ell^2$ points in $E[\ell]$ as 
\begin{equation}\label{ltorsion}
E[\ell] = \langle W_1 , W_2 \rangle = \{ P_\infty , \pm W_1 , \pm W_2 , \pm W_3 , \ldots , \pm W_{(\ell^2-1)/2} \} .
\end{equation}
For any $W_i\in E[\ell]$ we let $h_{W_i}(x,y)\in k(E)$  be a polynomial function with a single zero of order $\ell$ at $W_i$ 
\[
\divi(h_{W_i}) = \ell W_i - \ell P_\infty.
\]
The functions $h_{W_i}$ belong to the Riemann-Roch space 
\begin{equation*}
\mathcal{L}(\ell P_{\infty})=\mathrm{span}\{1,x,\ldots,x^{(\ell-1)/2}, y,\ldots,x^{(\ell -3)/2}y\}.
\end{equation*}
We normalize $h_{W_i}$ at infinity by setting the coefficient of $x^{(\ell -3)/2}y$ to be $1$. 

One finds the remaining $\ell-1$ coefficients of $h_{W_i}$  with Miller's algorithm \cite{M}. 
In this way $h_{W_i}$ is completely determined as, say  

\begin{equation}\label{defh}
h_{W_i}(x,y) = 
\begin{array}{rrcll}
a_{i0} &+\  a_{i1}x &+\ldots + &a_{i(\ell-1)/2}x^{(\ell-1)/2} \\
+\  b_{i0}y &+\  b_{i1}yx &+ \ldots +& b_{i(\ell -5)/2}yx^{(\ell -5)/2} &+\ yx^{(\ell -3)/2}
\end{array}
\end{equation}

The fact that $W_1,W_2$ are generators of the $\ell$-torsion implies 
the function $h_W(x,y)$ for any $W=aW_1+bW_2 \in E[\ell]=\langle W_1,W_2\rangle$ depends on $h_{aW_1}(x,y)$, $h_{bW_2}(x,y)$ and the line 
joining $aW_1$ and $bW_2$.

\begin{proposition}\label{relate}
If $W_i=(\gamma_i,\delta_i), W_j=(\gamma_j,\delta_j), W_k=(\gamma_k,\delta_k)\in E[\ell]$ with $W_k = W_i + W_j$, then
$$h_{W_k}(x,y) = h_{W_i}(x,y) \cdot h_{W_j}(x,y) \cdot \left( \frac{L_{i,j}(x,y)}{(x - \gamma_i)(x - \gamma_j)} \right)^\ell$$
where $L_{ij}=0$ is the line joining $W_i$ and $W_j$.
\end{proposition}
\begin{proof}
If $W_k = W_i + W_j$ then there exists a rational function $\epsilon_{ij}(x,y)$ such that
\[
W_k - P_{\infty} = (W_i - P_{\infty}) + (W_j - P_{\infty}) + \divi(\epsilon_{ij}(x,y))
\]
with
\[
\divi(\epsilon_{ij}(x,y)) = W_k - W_i - W_j + P_{\infty} \enspace .
\]
Since the line  $L_{ij}$ joining $W_i$ and $W_j$ cuts $E$ at $-W_k$, this suggests 
\begin{equation}\label{epsij}
\epsilon_{ij}(x,y) = \frac{L_{i,j}(x,y)}{(x-\gamma_i)(x-\gamma_j)}.
\end{equation}
We have 
\begin{align*}
\divi(h_{W_k}(x,y))&= \ell W_k  - \ell P_\infty = \ell \Big( W_k - P_\infty \Big)\\
&= \ell \bigg( \Big( W_k + (-W_i) + (-W_j) \Big) - (-W_i) - (-W_j) - P_\infty \bigg) \\
&= \ell \bigg( \Big( W_k + (-W_i) + (-W_j) \Big) - \Big( W_i + (-W_i) \Big) \\
&\quad\quad  - \Big( W_j + (-W_j) \Big)  + W_i + W_j - P_\infty \bigg) \\
&= \ell \bigg( \Big( W_k + (-W_i) + (-W_j) - 3P_\infty \Big) - \Big( W_i + (-W_i) - 2 P_\infty \Big) \\
&\quad\quad - \Big( W_j + (-W_j) - 2 P_\infty \Big) + \Big( W_i - P_\infty \Big) + \Big( W_j - P_\infty \Big) \bigg) \\
&= \left( \ell W_i - \ell P_\infty \right) + \left( \ell W_j - \ell P_\infty \right) + \ell\,\divi(L_{i,j})\\
& \quad\quad - \ell\,\divi(x-\gamma_i) - \ell \,\divi(x-\gamma_j) \\
&= \divi\left( h_{W_i}(x,y) \cdot h_{W_j}(x,y) \cdot \left( \frac{L_{i,j}(x,y)}{(x - \gamma_i)(x - \gamma_j)} \right)^\ell  \right).
\end{align*}
We conclude the functions $h_{W_k}(x,y)$ and $$ h_{W_i}(x,y)h_{W_j}(x,y)\left(\frac{L_{ij}(x,y)}{(x-\gamma_i)(x-\gamma_j)}\right)^{\ell}$$ are proportional. The equality follows due to our normalization at infinity. 
\end{proof}

\section{The function $g_P$}
If there exists $P\in E(k)$ such that $Q=\ell P$, then $\ell P - Q$ is a principal divisor and there is a nonzero polynomial function $g_P(x,y)\in k(E)$ such that
\[
\divi(g_P) = \ell P + (-Q) - (\ell+1) P_\infty.
\]
By Riemann-Roch we have 
\begin{equation*}
g_P(x,y)\in \mathcal{L}((\ell+1)P_{\infty})=\mathrm{span}\{1,x,\ldots,x^{(\ell+1)/2}, y,\ldots,yx^{(\ell -3)/2}\}.
\end{equation*}
Since $-Q$ is a fixed zero of $g_P$ we write

\begin{equation}\label{defg}
g_P(x,y) =
\begin{array}{rcl} 
\alpha_{1}(x-x_Q) &+ \ldots + & \alpha_{(\ell+1)/2}(x-x_Q)^{(\ell+1)/2}\\ 
+\  \beta_{0}(y+y_Q) &+ \cdots +& \beta_{(\ell -3)/2}(y+y_Q)(x-x_Q)^{(\ell -3)/2}
\end{array}
\end{equation}
where the coefficients $\alpha_i,\beta_j$ are unknowns until we find $P$. Our goal is to determine these as functions of $P$.

In Proposition \ref{findcoefsg} below we normalize $g_P$ at infinity by leaving the $x^{(\ell+1)/2}$ term monic. Notice, though, in case $\ell=2$ the functions $y$ and $x^{3/2}$ have the same valuation at infinity and one may choose the monic term. 
For our convenience, we let $$r=\frac{\ell+1}{2}$$ and, for $\ell\ge 5$, we abbreviate the $\beta$-terms of $g_P$ as
$$y_i=(y+y_Q)(x-x_Q)^{i},\,i=0,\ldots,(\ell-1)/2=r-1.$$ We write the value of $y_i$ at $P$ just as $y_i$.
\begin{proposition}\label{findcoefsg}
If there is $P\in E(k)$ such that $\ell P = Q$ then the coefficients of $g_P$ are the solution of the homogeneous linear equation 
\begin{equation}\label{eqgp}
MX=0
\end{equation}
 where $M$ is the singular matrix 
\begin{equation*}
\begin{pmatrix}
x_P-x_Q & (x_p-x_q)^2 &\ldots &(x_P-x_Q)^r & y_P+y_Q& \cdots & y_{r-1}\\
1& 2(x_P-x_Q) &\ldots  &r(x_P-x_Q)^{r-1} & y'_P & \cdots & y_{r-1}'\\
0& 2&\ldots  &\vdots & y_P'' & \cdots & y_{r-1}''\\
\vdots&\vdots&\ddots&\vdots&\vdots&\vdots&\vdots\\
0& 0&  &r!& y_P^{(r)} & \cdots & y_{r-1}^{(r)}\\
\vdots&\vdots&\ddots&\vdots&\vdots&\vdots&\vdots\\
0 &\ldots  &\ldots &0 & y_P^{(\ell-1)} & \cdots & y_{r-1}^{(\ell-1)}
\end{pmatrix}.
\end{equation*}
\end{proposition}
\begin{proof}
If $P$ exists, then by evaluating $g_P$ and its $\ell-1$ derivatives with respect to $x$ at $P$ 
we obtain the homogeneous linear  $\ell\times \ell$ equation above. 
We deduce $\det{M}=0$ because the variables $\alpha_i,\beta_j$ cannot be all equal to $0$.   
The solutions of \eqref{eqgp} are one dimensional, but the normalization at infinity $\alpha_{r}=1$ 
leaves all $\alpha_i,\beta_j$ univocally determined as functions of $P$.
\end{proof}

\begin{lemma}\label{dy}
The iterated derivatives of $y$ satisfy
\begin{equation*}
\begin{array}{lr}
2yy'=3x^2+a,&\\
y'^2+yy''=3x,&\\
3y'y''+ yy'''=3,&\\
&\\
\displaystyle\sum\limits_{i=0}^{(n-2)/2}\binom{n}{i}y^{(i)}y^{(n-i)} + \binom{n-1}{(n-2)/2}{y^{(\frac{n}{2})}}^2=0&\text{ if $n\ge 4$ is even,}\\ 
&\\
\displaystyle\sum\limits_{i=0}^{(n-1)/2}\binom{n}{i}y^{(i)}y^{(n-i)}=0&\text{ if $n\ge 5$ is odd}.
\end{array}
\end{equation*}
\end{lemma}
\begin{proof}
By induction on the derivation index.
\end{proof}

\begin{remark}
From the definition of the derivatives of $y$, the right-hand side of matrix $M$ in Proposition~\ref{findcoefsg} would be rational functions with denominators of the form $y_P^k$.
For computational purposes, since we want elements in the kernel of the matrix, it is advisable to multiply the rows of $M$ by the appropriate power of $y_P$ so the entries are all polynomial in $x_P$ and $y_P$ and compute the kernel over polynomials in $k[x_P,y_P]$.
\end{remark}

\begin{example}\label{ell2}
Let $\ell=2$. 
If $(0,0)\in E[2]$, then $2P=Q$ with $P\in E(k)$ implies $x(Q)$ is a square of $k$ as shown in \cite[p. 87 (Prop. 3.7)]{ST}. 
Of course $x(Q)= x(-Q)-x((0,0))$. More generally, if $\gamma_i$ are the roots of the $2$-division polynomial $\psi_2(x)$ of $E$, and   
then all $x(Q)-\gamma_i$ are squares of $k$ by \cite[Th. 4.1 on p. 38]{Hu}. In our notation $W_i=(\gamma_i,0)$ we have 
$$h_{W_i}(x,y)=x-\gamma_i$$ and $$\omega_i=\sqrt{x_Q-\gamma_i}.$$
If $2P=Q$ then there exists a function $g_P$ on $E$ of the type
$$g_P(x,y)=\alpha_1(x-x_Q)+ \beta_0(y+y_Q).$$ 
The matrix $M$ in \eqref{eqgp} is
\begin{equation*}
M=\begin{pmatrix}
x_P-x_Q&y_P+y_Q\\
1& y'_P\\
\end{pmatrix}
\end{equation*}
and  
\begin{equation*}
\det{M}=-(y_P+y_Q)+y'_P(x_P-x_Q)=0.
\end{equation*} 
The condition $\det{M}=0$ is equivalent to saying $y(-Q)$ is equal to the 1st order Taylor expansion of $y$ near $P$ evaluated at $x_Q$. 

Solving for $\alpha_1$ with $\beta_0=1$ we obtain 
$$\alpha_1(x_P,y_P)=\frac{y_P+y_Q}{x_Q-x_P}=-y'_{P}$$
and then 
\begin{equation}\label{gp2}
g_P(x,y)=-y_P'(x-x_Q)+ (y+y_Q).
\end{equation}
\end{example}

\begin{example}\label{ell3}
Let $\ell=3$, let $\gamma_i$ be the four different roots of the $3$-division polynomial $\psi_3(x)$ and let $W_i=(\gamma_i,\delta_i)\in E[3](k)$.
A basis of the $3$-torsion $\langle W_1,W_2\rangle =E[3](k)$ is given by any pair $W_1=(\gamma_1,\delta_1), W_2=(\gamma_2,\delta_2)$ with $\gamma_1\ne\gamma_2$. The authors of \cite{PT} noticed $y(3P)$ is a cube of $k$ ``with respect to $(0,0)$'' -- this is, in a model where $(0,0)$ is a point of order $3$. The full Miller functions returning cubes are $h_{(0,0)}(x,y)=y$ and, more generally $$h_{W_i}(x,y)=y-\delta_i-m_i(x-\gamma_i)$$ where $m_i=y'_{W_i}$ (see \cite{PT} for details).  In our notation, the coefficients of $h_{W_i}(x,y)$ are
$a_{i1}=-m_i$ and $a_{i0}=\gamma_i m_i-\delta_i$.

The function $g_P(x,y)$ on $E$ is of the form
\begin{equation*}
g_P(x,y)=\alpha_1(x-x_Q)+ \alpha_2(x-x_Q)^2+\beta_0(y+y_Q),
\end{equation*} 
the matrix $M$ in \eqref{eqgp} is
\begin{equation*}
M=
\begin{pmatrix}
x_P-x_Q&(x_P-x_Q)^2&y_P+y_Q\\
1&2(x_P-x_Q)& y'_P\\
0&2& \dfrac{\psi_3(x_P)}{4y_P^3}
\end{pmatrix}
\end{equation*}
and 
\begin{equation*}
\begin{array}{cl}
\det{M}=&\dfrac{y_P''}{2}(x_Q-x_P)^2+y_P'(x_Q-x_P)+y_P+y_Q=0 \enspace ,
\end{array}
\end{equation*}
or equivalently (over $k[x_P,y_P]$):
\[
\widetilde{M} =
\begin{pmatrix}
x_P-x_Q&(x_P-x_Q)^2&y_P+y_Q\\
1&2(x_P-x_Q)& y'_P\\
0&2y_P^3& \frac14 \psi_3(x_P)
\end{pmatrix}
\]

Solving with 
$\alpha_2=1$ we find
\begin{equation*}
\beta_0=\frac{-8y_P^3}{\psi_3(x_P)}=-\frac{2}{y_P''},\,\,\alpha_1
=2(x_Q-x_P)+2\frac{y_P'}{y_P''}
\end{equation*}
and then our $g_P$ is 
\begin{equation}\label{gp3}
\begin{aligned}
g_P(x,y)=\left(2(x_Q-x_P)+2\dfrac{y_P'}{y_P''}\right)(x-x_Q)+(x-x_Q)^2-\dfrac{2}{y_P''}(y+y_Q)
\end{aligned}
\end{equation}
By the recurrence in Lemma \ref{dy}, one may write 
$g_P(x,y)$ in terms of $y_P$ and no derivative.
\end{example}

\begin{example}\label{ell5}
Let $\ell=5$, let $\gamma_i$ be the four different roots of the $5$-division polynomial $\psi_5(x)$ and let $W_i=(\gamma_i,\delta_i)\in E[5](k)$.
A basis of the $5$-torsion $\langle W_1,W_2\rangle =E[5](k)$ is given by any pair $W_1=(\gamma_1,\delta_1), W_2=(\gamma_2,\delta_2)$ with $W_2 \not\in \{ \pm W_1, \pm 2W_1 \}$.

The function $g_P(x,y)$ on $E$ is of the form
\begin{equation*}
g_P(x,y)=\alpha_1(x-x_Q)+ \alpha_2(x-x_Q)^2+ \alpha_3(x-x_Q)^3+\beta_0(y+y_Q)+\beta_1(x-x_Q)(y+y_Q),
\end{equation*} 
the matrix $M$ in \eqref{eqgp} is
\[
M=
\begin{pmatrix}
x_P-x_Q & (x_P-x_Q)^2 & (x_P-x_Q)^3 & y_P+y_Q & (x_P-x_Q)(y_P+y_Q)\\
1 & 2 (x_P-x_Q) & 3 (x_P-x_Q)^2 & y_P' & (x_P-x_Q) y_P' + (y_P+y_Q) \\
0 & 2 & 6 (x_P-x_Q) & y_P'' & (x_P-x_Q) y_P'' + 2 y_P' \\
0 & 0 & 6 & y_P^{(3)} & (x_P-x_Q) y_P^{(3)} + 3 y_P'' \\
0 & 0 & 0 & y_P^{(4)} & (x_P-x_Q) y_P^{(4)} + 4 y_P^{(3)}
\end{pmatrix}
\]
with $\det{M}=0$,
or equivalently (over $k[x_P,y_P]$):
{\tiny
\[
\widetilde{M}=
\begin{pmatrix}
x_P-x_Q & (x_P-x_Q)^2 & (x_P-x_Q)^3 & y_P+y_Q & (x_P-x_Q) (y_P+y_Q)\\
2 y_P & 4 (x_P-x_Q) y_P & 6 (x_P-x_Q)^2 y_P & r_2(x_P) & (x_P-x_Q) r_2(x_P) + 2 (y_P+y_Q) y_P \\
0 & 8 y_P^3 & 24 (x_P-x_Q) y_P^3 & r_3(x_P) & (x_P-x_Q) r_3(x_P) + 4 y_P^2 r_2(x_P) \\
0 & 0 & 192 y_P^5 & -3 r_4(x_P) & -3 (x_P-x_Q) r_4(x_P) + 24 y_P^2 r_3(x_P) \\
0 & 0 & 0 & r_5(x_P) & (x_P-x_Q) r_5(x_P)  - 6 y_P^2 r_4(x_P) 
\end{pmatrix}
\]
}
with
{\scriptsize
\begin{align*}
r_2(x) &= 3x^2+a = f'(x) \\
r_3(x) &= 3x^4+6ax^2+12bx-a^2 = \psi_3(x) \\
r_4(x) &= 4x^6+20ax^4+80bx^3-20a^2x^2-16abx-4a^3-32b^2 = \psi_4(x)\\
r_5(x) &= 9x^8+84ax^6+504bx^5-210a^2x^4-336abx^3-60a^3x^2-720b^2x^2+24a^2bx-15a^4-96ab^2 .
\end{align*}
}
Computing an element of the kernel with $\alpha_3=1$ gives $g_P(x,y)$.
\end{example}

\section{Division by $\ell$ and $\ell$-powers}
In this Section we use the functions $h_{W_i}$ and $g_P$ to prove Theorem \ref{main}. 
As a first step in this direction, Proposition \ref{char} below shows the pre-images of $[\ell]$ in $E(k)$ are related to $\ell$-powers in $k$. 
\begin{proposition}\label{char}
Let $Q\in E(k)$ and $W_i\in E[\ell]$. 
If there is $P\in E(k)$ such that $\ell P=Q$ then $h_{W_i}(-Q)$ is an $\ell$-power in $k$ for every function 
$h_{W_i}(x,y),\, i=1,\ldots,\ell^2-1$ in \eqref{defh}.
\end{proposition}
\begin{proof}
We have
\[
h_{W_i}(x,y) \Big|_{(-Q)+\ell P} = h_{W_i}(-Q) \cdot \Big( h_{W_i}(P) \Big)^\ell
\]
and by Weil reciprocity we obtain  
\begin{align*}
h_{W_i}(x,y) \Big|_{(-Q)+\ell P} &= h_{W_i}(x,y) \Big|_{\divi(g_P)} = g_P(x,y) \Big|_{\divi(h_{W_i})} \\
&= g_P(x,y) \Big|_{\ell W_i} \qquad\ \! = \Big( g_P(W_i) \Big)^\ell.
\end{align*}
Hence 
\begin{equation}\label{sys}
\Big( g_P(W_i) \Big)^\ell = h_{W_i}(-Q) \cdot \Big( h_{W_i}(P) \Big)^\ell.
\end{equation}

Alternatively, one can argue with the product formula 
$$\prod_{T\in E}(h_{W_i},g_P)_T=1 $$
for local symbols 
$$(h_{W_i},g_P)_T=(-1)^{mn}\frac{h_{W_i}^n}{g_P^m}$$ 
at points $T\in E$ with $m=v_T(h_{W_i}), n=v_T(g_{P})$ (see \cite[pg. 34]{Se}).
The computation is reduced to the points $T=-Q,P, W_i$ and $P_{\infty}$, and we obtain 
\begin{equation*}
\dfrac{h_{W_i}(-Q)\cdot h_{W_i}(P)^{\ell}\cdot g_P(P_{\infty})^{\ell}}{g_P(W_i)^{\ell}\cdot h_{W_i}(P_{\infty})^{\ell+1}}=1.
\end{equation*}
The upper and lower rightmost factors cancel because both $h_{W_i}$ and $g_P$ are monic at infinity, and our claim 
follows for every $i=1,\ldots,\ell^2-1$.
In any case if there is $P\in E(k)$ such that $\ell P=Q$ then  
\begin{equation}\label{lpow}
h_{W_i}(-Q)=\Big{(}\dfrac{g_P(W_i)}{h_{W_i}(P)}\Big{)}^{\ell}
\end{equation} is an $\ell$-power for all $i=1,\ldots,\ell^2-1$.
\end{proof}

\begin{definition}\label{lsector}
We call any $\omega_i \in k$ such that $\omega_i^\ell = h_{W_i}(-Q)$ an \emph{$\ell$-sector of $Q$} with respect to $W_i$.
\end{definition}

Clearly, if $\omega_i$ is an $\ell$-sector, then $\zeta \omega_i$ are all $\ell$-sectors of $Q$ with respect to $W_i$ for any $\zeta\in k$ such that $ \zeta^{\ell}=1$. 
Because of Proposition \ref{char}, we deduce that every solution $P$ to $Q = \ell P$ should correspond to a choice of an appropriate $\zeta_i$ such that 
\[
\zeta_{i}\omega_i \cdot h_{W_i}(P) = g_P(W_i).
\]
has a solution for every $i$. Naturally, for every $P$ the choice of $\zeta_{i}$ in two ``generator'' equations should be sufficient to determine the correct $\zeta_{i}$ for all other $i$'s. With this purpose we fix two $\ell$-sectors $\omega_1,\omega_2\in k$ of $Q$ with respect to two generators $W_1,W_2$ of $E[\ell]$, and we consider the equations   
\begin{equation}\label{linsysoverdetermined}
\left\{
\begin{array}{rcl}
e_1:&g_P(W_1)-\omega_{1}h_{1}(P)&=0\\
e_2:&g_P(W_2)-\omega_{2}h_{2}(P)&=0
\end{array}
\right.
\end{equation}

In the next proposition we compute an $\ell$-sector with respect to $W_i+W_j$ in terms of $\omega_i,\omega_j$. 
This allows to deduce $\ell$-sectors with respect to all $\langle W_i,W_j\rangle$, and thus obtain equations that are consistent with the group law.

\begin{proposition}\label{rightomegas}
Let $W_i, W_j, W_k\in E[\ell]$. If $W_k = W_i + W_j$ then 
\[
\omega_k = \omega_i \cdot \omega_j \cdot \frac{L_{i,j}(x_Q,-y_Q)}{(x_Q - \gamma_i)(x_Q - \gamma_j)} .
\]
\end{proposition}
\begin{proof}
Let $P$ be a solution of \eqref{linsysoverdetermined} 
for $\ell$-sectors $\omega_i, \omega_j$ and let
$$\epsilon_{ij}(x,y)= \frac{L_{i,j}(x,y)}{(x-\gamma_i)(x-\gamma_j)}$$ as in \eqref{epsij} above.
By Proposition \ref{relate} then
\begin{equation*}
\epsilon_{ij}^{\ell}(x,y)=\frac{h_{W_k}(x,y)}{h_{W_i}(x,y)h_{W_j}(x,y)}.
\end{equation*}
But then
\begin{align*}
\omega_k &= \sqrt[\ell]{h_{W_k}(-Q)} 
= \frac{g_P(W_k)}{h_{W_k}(P)} 
= \frac{g_P(W_i) \cdot g_P(W_j) \cdot g_P(\divi(\epsilon_{ij}))}{h_{W_i}(P) \cdot h_{W_j}(P) \cdot \frac{h_{W_k}}{h_{W_i} h_{W_j}}(P)} \\
&= \omega_i \cdot \omega_j \cdot \frac{g_P(\divi(\epsilon_{ij}))}{\frac{h_{W_k}}{h_{W_i} h_{W_j}}(P)} 
= \omega_i \cdot \omega_j \cdot \frac{\epsilon_{ij}(\divi(g_p))}{\frac{h_{W_k}}{h_{W_i} h_{W_j}}(P)} \\
&= \omega_i \cdot \omega_j \cdot \frac{\epsilon_{ij}(\ell P + (-Q) - (\ell+1) \infty)}{\frac{h_{W_k}}{h_{W_i} h_{W_j}}(P)} 
= \omega_i \cdot \omega_j \cdot \frac{\epsilon_{ij}(P)^\ell \cdot \epsilon_{ij}(-Q)}{\frac{h_{W_k}}{h_{W_i} h_{W_j}}(P)} \\
&= \omega_i \cdot \omega_j \cdot \epsilon_{ij}(-Q) \cdot \frac{\epsilon_{ij}(P)^\ell}{\frac{h_{W_k}}{h_{W_i} h_{W_j}}(P)}
= \omega_i \cdot \omega_j \cdot \epsilon_{ij}(-Q).
\end{align*}
\end{proof}

\subsection{A characterization of $[\ell]E(k)$}
In the remaining of this Section we discuss the solutions of \eqref{linsysoverdetermined}. We first study how different equations $e_1,e_2$ are and then the unicity of the solutions. Finally we find a solution.

Let $\tau_{P}$ be the translation-by-$P$ map  on $E$ $\tau_{P}(x,y):(x,y)\mapsto (x,y)+P$  and 
let $s_{P+W}(x,y)=h_{W}\circ \tau_P$. Then 
\begin{equation*}
\divi(s_{P+W})=\ell (P+W) - \ell P.
\end{equation*}

\begin{lemma}\label{moresolutions}
Let $W,W'\in E[\ell]$. Then
\begin{equation*}
\begin{aligned}
g_{P+W}(W') & =  g_P(W')s_{P+W}(W'),\\
h_{W'}^{\ell}(P+W) & =  h_{W'}^{\ell}(P)s_{P+W}^{\ell}(W').
\end{aligned}
\end{equation*}
\begin{proof}
The first equality is immediate because

\begin{align*}
\divi(g_P)+\divi(s_{P+W})&=(-Q)+ \ell P - (\ell+1) P_{\infty}+\ell (P+W) - \ell P\\
&=(-Q)+ \ell (P+W) - (\ell+1) P_{\infty}=\divi(g_{P+W}).
\end{align*}
The second follows by Weil reciprocity
\begin{align*}
\dfrac{h_{W}^{\ell}(P+W)}{h_{W'}^{\ell}(P)} &=h_{W'}(\ell(P+W)-\ell P)=s_{P+W}(\ell W' - \ell P_{\infty}) \\
&=\dfrac{s_{P+W}^{\ell}(W')}{s_{P+W}^{\ell}(P_{\infty})}=s_{P+W}^{\ell}(W')
\end{align*}
\end{proof}
\end{lemma}

By Lemma \ref{moresolutions} we deduce the existence of an $\ell$-th root of unity $\zeta$ such that
\begin{equation}\label{transfgh}
\begin{aligned}
g_{P+W}(W')  &=  g_P(W')s_{P+W}(W'),\\
h_{W'}(P+W)  &=  h_{W'}(P)s_{P+W}(W')\zeta.
\end{aligned}
\end{equation}

\begin{lemma}\label{zetaeq1}
If $W=W'$ then $\zeta=1$ in \eqref{transfgh}.
\begin{proof}
We have
\begin{equation*}
s_{P+W}(W-P_{\infty})=h_{W}(\tau_{P}(W-P_{\infty}))=h_{W}((W+P)-P)
\end{equation*}
hence 
$$\dfrac{h_{W}(P+W)}{h_{W}(P)}=\dfrac{s_{P+W}(W)}{s_{P+W}(P_{\infty})}$$
as desired.
\end{proof}
\end{lemma}

\begin{corollary}\label{zetaeq1subgroup} 
If $W'\in \langle W\rangle$ then $\zeta=1$ in \eqref{transfgh}.
\begin{proof}
Write $W'=aW$. By Lemma \ref{zetaeq1} above,   
\begin{equation*}
h_{W'}(P+W)=h_{aW}(P+W)=h_{aW}(P)s_{P+W}(aW)\zeta = h_{W}(P)^as_{P+W}(aW)\zeta.
\end{equation*}

The relation $h_{aW}(P)=h_{W}(P)^a$ holds as long as $P$ is not in the support of $h_{aW}$ nor $h_W$. Therefore 
\begin{equation*}
h_{W'}(P+W)=h_{W}(P+W)^a=(h_{W}(P)s_{P+W}(W))^a=h_W(P)^a s_{P+W}(W)^a
\end{equation*}
as well, and we deduce $s_{P+W}(aW)\zeta=s_{P+W}(W)^a$.

But clearly $s_{P+W}(aW)=s_{P+W}(W)^a$. Hence $\zeta=1$ as claimed.
\end{proof}
\end{corollary}

By Corollary \ref{zetaeq1subgroup}, if $P$ satisfies $g_P(W')=\omega h_{W'}(P)$ with $W'\in \langle W \rangle$ then $P+W$ satisfies the same relation because then the $\zeta$ in \eqref{transfgh} keeps being $1$ and the equations defined by $h_{W'}$ and $h_{W}$ are proportional:
\begin{equation*}
g_{P+W}(W')=g_P(W')s_{P+W}(W')=\omega h_{W'}(P)s_{P+W}(W')=\omega h_{W'}(P+W).
\end{equation*}

Next we show that not all points of the form $P+aW_1+bW_2$ are solutions of the 1st equation of \eqref{linsysoverdetermined}.
\begin{lemma}\label{notall}
Let $P\in E(k)$ be a solution of \eqref{linsysoverdetermined}. Then not all $P+W$, for $ W\in E[\ell]$, can be a solution of the 1st equation $e_1$ in \eqref{linsysoverdetermined}.
\begin{proof}
Assume on the contrary that all points $P+W, \forall W\in E[\ell]$ satisfy the first equation
\begin{equation}\label{periodic}
\dfrac{g_{P+W}(W_1)}{h_{W_1}(P+W)}=\omega_1.
\end{equation}
If we let $G_{W_1}$ be the function
\begin{equation*}
G_{W_1}(P+W)=\dfrac{g_{P+W}(W_1)}{h_{W_1}(P+W)},
\end{equation*}
then our assumption in terms of $G_{W_1}$ is 
\begin{equation*}
G_{W_1}(P+W)=G_{W_1}(P), \forall W\in E[\ell].
\end{equation*}
By an argument similar to that in the proof of the non degeneracy of the Weil pairing in \cite[pg. 97]{S}, there exists a function $G'$ on $E$ such that
$$G_{W_1} = G'\circ [\ell].$$ 
This implies
\begin{align*}
(G' \circ [\ell])^{\ell}(P) &= G_{W_1}^{\ell}(P)=\Big{(}\dfrac{g_{P+W}(W_1)}{h_{W_1}(P+W)}\Big{)}^{\ell} \\
&= h_{W_1}(-Q) = h_{W_1}(-\ell P) = h_{W_1} \circ \iota \circ [\ell](P)
\end{align*}
by Theorem \ref{char}, where $\iota$ is the hyperelliptic involution on $E$. By \cite[pg. 97]{S} we may cancel the rightmost $[\ell]$ and deduce
$$(G')^{\ell}=h_{W_1}\circ \iota,$$
implying
$\ell \, \divi(G')= \divi(G')^{\ell}=\ell (-W_1)-\ell P_{\infty}$.
Hence 
$$\divi(G')=(-W_1)-P_{\infty}$$ 
which is a contradiction.
\end{proof}
\end{lemma}

\begin{corollary}\label{e1e2np}
If $W_1,W_2$ are a basis of $E[\ell]$ then equations $e_1,e_2$ in \eqref{linsysoverdetermined} are not proportional.
\end{corollary}
\begin{proof}
This follows by Corollary \ref{zetaeq1subgroup} and Lemma \ref{notall}.
\end{proof}

We now show there is at most one solution to equation \eqref{linsysoverdetermined} of the form $P+W$ with $W \in E[\ell]$. For this, let 
\begin{equation*}
\begin{array}{cccl}
e_{P,W_1}:&E[\ell] &\longrightarrow&\mu_{\ell}\\
&W&\longmapsto&\dfrac{g_{P+W}(W_1)}{h_{W_1}(P+W)}/\dfrac{g_{P}(W_1)}{h_{W_1}(P)}=\dfrac{G_{W_1}(P+W)}{G_{W_1}(P)}
\end{array}
\end{equation*}
\begin{lemma}
The map 
$e_{P,W_1}(W)$ is a surjective homomorphism.
\begin{proof}
We first show 
$$e_{P,W_1}(W'+W'')=e_{P,W_1}(W') \cdot  e_{P,W_1}(W'').$$
We have
\begin{align*}
e_{P,W_1}(W'+W'') &=\dfrac{g_{P+W'+W''}(W_1) h_{W_1}(P)}{h_{W_1}(P+W'+W'')g_{P}(W_1)}\\
&=\dfrac{g_{P+W'+W''}(W_1) h_{W_1}(P+W')}{h_{W_1}(P+W'+W'') g_{P+W'}(W_1)} \,\cdot\, \dfrac{g_{P+W'}(W_1)h_{W_1}(P)}{h_{W_1}(P+W') g_{P}(W_1)}\\
&= \dfrac{g_{P+W'+W''}(W_1) h_{W_1}(P+W')}{h_{W_1}(P+W'+W'') g_{P+W'}(W_1)} \,\cdot\,e_{P,W_1}(W')
\end{align*}

Hence our claim follows if 
\begin{equation*}
\dfrac{g_{P+W'+W''}(W_1) h_{W_1}(P+W')}{h_{W_1}(P+W'+W'') g_{P+W'}(W_1)}= e_{P+W',W_1}(W'') \overset{?}{=}e_{P,W_1}(W'')
\end{equation*}
But this is true because like the Weil pairing, $e_{P,W_1}$ does not depend on $P$.

We deduce $e_{P,W_1}$ is a nontrivial morphism by Lemma \ref{notall}.
Hence $e_{W_1,P}$ is surjective because $\mu_{\ell}$ has no proper subgroups since $\ell$ is prime.
By Lemma \ref{zetaeq1}, $\langle W_1\rangle \subseteq ker(e_{W_1,P})$, and we conclude $E[\ell]/\langle W_1\rangle\cong \mu_{\ell}$.
\end{proof}
\end{lemma}

\begin{corollary}\label{uniquemodl}
If \eqref{linsysoverdetermined} has a solution $P\in E(k)$ then it has no other solution of the form $P + W, W\in E[\ell]\setminus P_{\infty}$.
\end{corollary}

By the theory of division polynomials $\psi_{n}(x)$, we have
\[
x_{\ell P} = x_P - \frac{\psi_{\ell-1}(x_P) \psi_{\ell+1}(x_P)}{\psi_\ell(x_P)^2}
\]
(see \cite[pg.105]{S}).
As a consequence, $\{x(P)\mid \ell P=Q \}$ is the zero set of 
\begin{equation}\label{polyanti}
\phi_{\ell,Q}(x)=(x-x_Q)\psi_l(x)^2-\psi_{l-1}(x)\psi_{l+1}(x)\in k[x].
\end{equation}
We next prove our characterization.

\begin{theorem}\label{main}
If $Q\in E(k)$, then there exists $P\in E(k)$ such that $\ell P = Q$ if and only if $h_{W_1}(-Q)$ and $h_{W_2}(-Q)$ are $\ell$-powers in $k$.
\end{theorem}

\begin{proof}
The forward implication is given by Proposition \ref{char}. 
For the converse to hold, we assume  $h_{W_1}(-Q)$ and $h_{W_2}(-Q)$ are $\ell$-powers in $k$ and we need to find one solution $P=(x_P,y_P)\in E(k)$
to $\ell P=Q$. Let  $\omega_1,\omega_2\in k$ two $\ell$-sectors with respect to $W_1$ and $W_2$ respectively. We will find $P$ by solving the bivariate equations \eqref{linsysoverdetermined}. 
The unknowns $x_P,y_P$ in \eqref{linsysoverdetermined} appear in the variables of 
$h_{W_i}(x_P,y_P)$ and also 
in the formulas of the coefficients $\alpha_i(x_P,y_P),\,\beta_j(x_P,y_P)$  as rational 
functions of $x_P,y_P$ provided by Proposition \ref{findcoefsg}. These all have the form $s(x_P)+t(x_P)y_P$ for some rational functions $s(x),t(x)$ by Lemma \ref{dy}. 

We deduce the degree in $y_P$ cannot be $0$ in both equations $e_1,e_2$ because these 
are associated to an independent set of generators $W_1,W_2$ of $E[\ell]$. By Corollary \ref{e1e2np} in this case $e_1,e_2$ are not proportional. Hence  \eqref{linsysoverdetermined} is nondegenerate in $y_P$, and by equating the expressions of $y_P$ 
we obtain a polynomial equation $\varphi(x_P)=0$ of strictly positive degree and coefficients in $k$,  with a root $x_P$.
But $x_P$ is also a root of the polynomial $\phi_{\ell,Q}(x)\in k[x]$ in \eqref{polyanti} above. Hence $x_P$ is a root of 
$gcd(\varphi(x),\phi_{\ell,Q}(x))=0$. Moreover, $x_P$ is the only common root of both polynomials by Corollary \ref{uniquemodl}.
Since  $\varphi(x),\phi_{\ell,Q}(x)$ have coefficidents in $k$, this means $gcd(\varphi(x),\phi_{\ell,Q}(x))=(x-x_P)\in k[x]$. Hence $x_P\in k$
and, together with the corresponding $y_P$, we found  a point $P\in E(k)$ which satisfies $\ell P=Q$ by construction.
\end{proof}

\begin{corollary}\label{bij}
There is a bijection between $\{P\in E(k)\mid \ell P=Q\}$ and $\{(\omega_1,\omega_2)\mid \omega_i^{\ell}=h_{W_i}(-Q), \, i=1,2\}$.
\end{corollary}

\subsection{Computation of pre-images}

In the final part of this Section we turn the characterization of Theorem \ref{main} into an effective procedure to find pre-images of $[\ell]$. Given $Q\in E(k)$ and a pair of $\ell$-sections $(\omega_1,\omega_2)$ associated to a basis of $E[\ell]$, we establish a constructive method to find $P\in E(k)$ such that $\ell P=Q$. Details of explicit computations for $\ell=2,3$ are provided in our final examples.

\begin{corollary}\label{gcd}
Let $Q\in E(k)$, and let $\phi_{\ell,Q}(x)$ as in \eqref{polyanti} above. For any pair of $\ell$-sectors $(\omega_1,\omega_2)$ with respect to a basis $\{W_1,W_2\}$ of $E[\ell]$, we construct a polynomial $\varphi(x)$ (with $\omega_1,\omega_2$ in the coefficients) such that  the $x$-coordinate of a pre-image of $Q$ by the multiplication-by-$\ell$ map on $E$ is the unique root of $gcd(\varphi(x),\phi_{\ell,Q}(x))$.
\end{corollary}

\begin{proof}
Writing the function $g_P$ in terms of $x_P$ and $y_P$, then \eqref{linsysoverdetermined} has the form 
\begin{equation*}
\begin{array}{l}
e_i(x_P,y_P):  
A_i(x_P)-B_i(x_P)y_P=0, \quad i=1,2
\end{array}
\end{equation*}
for certain univariate polynomials $A_i,B_i$ with $\omega_1,\omega_2$ in the coefficients. Since $e_1$ and $e_2$ are not proportional, it makes sense to define 
\[
\varphi(x) = A_2(x) B_1(x) - B_2(x) A_1(x) \enspace .
\]
We deduce $x_P$ is a common root of $\varphi(x)$ and $\phi_{\ell,Q}(x)$.
By Corolary~\ref{bij}, the root is unique.
\end{proof}

\begin{example}
Let $\ell=2$. After Exemple \ref{ell2}, if we evaluate our $g_P$ in \eqref{gp2} at $W_i=(\gamma_i,0)$ for $i=1,2$ 
then 
\eqref{linsysoverdetermined} is linear in $x_P, y'_P$: 
\begin{equation*}
\begin{array}{l}
e_i(x_P,y_P): \, -y_P'(\gamma_i-x_Q)+y_Q - \omega_i(x_P-\gamma_i)=0,\quad i=1,2.
\end{array}
\end{equation*}
The solution
\begin{equation*}
\begin{aligned}
x_P=&\frac{y_Q(\gamma_2-\gamma_1)+\gamma_2\gamma_1(\omega_1-\omega_2)-x_Q(\omega_1\gamma_1-\omega_2\gamma_2)}{(\omega_1 - \omega_2)x_Q+\omega_1\gamma_2-\omega_2\gamma_1}, \\
y_P'=&\frac{y_Q(\omega_1-\omega_2)+\omega_1\omega_2(\gamma_2-\gamma_1)}{(\omega_1 - \omega_2)x_Q+\omega_1\gamma_2-\omega_2\gamma_1}
\end{aligned}
\end{equation*}
determines $y_P$ by Lemma \ref{dy}.
\end{example}

\begin{example}
Let $\ell=2$ and let $$E: y^2=x^3 - 1323x + 3942$$ over $\mathbb{Q}$. 
The full $2$-torsion is defined over $\mathbb{Q}(\xi)$ with $\xi^2 + 3\xi - 1314=0$ and a basis is $W_1=(3,0), W_2=(\xi,0)$.  Let $\zeta_1,\zeta_2\in\{-1,1\}$ and let $Q=(147,1728)\in E(\mathbb{Q})$. Then 
\begin{align*}
\varphi(x)&=\big((72\xi - 10584)\zeta_1 + (288\xi - 864)\zeta_2\big)x^3 \\
&\quad+ \big((-216\xi + 31752)\zeta_1 +  (1728\xi - 378432)\zeta_2 - 10368\xi + 31104\big)x^2 \\
&\quad+ \big((-31752\xi + 4667544)\zeta_1 + (-127008\xi + 381024)\zeta_2\big)x \\
&\quad+ (95256\xi - 14002632)\zeta_1 + (-762048\xi + 166888512)\zeta_2 + 4572288\xi - 13716864\\
&=\big((72\xi - 10584)\zeta_1 + (288\xi - 864)\zeta_2\big)(x^2-21^2)\\
&\quad \cdot \left(x - \dfrac{108\zeta_1 - 3(\xi + 438)\zeta_2 - 36 (\xi + 6)}{36\zeta_1 - (\xi + 6)\zeta_2}\right)
\end{align*}
and the last factor at the different values of $\zeta_1,\zeta_2\in\{-1,1\}$ provides the $4$ preimages  
$$(39 , -108),\quad (-33 , -108 ), \quad (-8\xi + 279 , 192\xi - 6732),\quad (8\xi + 303 , 192\xi + 7308)$$ of $Q$.
\end{example}

\begin{example}
Let  $\ell=3$. If we evaluate our $g_P$ in \eqref{gp3} at two $3$-torsion points $W_i=(\gamma_i,\delta_i)\in E[3]$ then 
\eqref{linsysoverdetermined} is
\begin{equation}\label{sys3}
\begin{aligned}
(2(x_Q-x_P)+2\dfrac{y_P'}{y_P''})(\gamma_i-x_Q)+(\gamma_i-x_Q)^2&\\
-\dfrac{2}{y_P''}(\delta_i+y_Q)-\omega_i(\gamma_i m_i -\delta_i-m_ix_P+y_P)&=0
\end{aligned}
\end{equation} for $i=1,2$. By cleaning denominators and evaluating $g_P$ at $W_i$, then the linear and constant terms in $y_P$  are
\begin{equation}\label{AsBs}
\begin{aligned}
B_i(x_P)&=8f(x_P)(\delta_i+y_Q)+ \omega_i\psi_3(x_P)\\
A_i(x_P)&=\big{(}2\psi_3(x_P)(x_Q-x_P)+4f(x_P)f'(x_P)\big{)}(\gamma_i-x_Q) \\
&\quad + \psi_3(x_P)(\gamma_i-x_Q)^2 - \omega_i\psi_3(x_P)(\gamma_i m_i -\delta_i-m_ix_P)
\end{aligned}
\end{equation}
so that  \eqref{sys3} has the form 
\begin{equation*}
\begin{array}{l}
e_i(x_P,y_P):  
A_i(x_P)-B_i(x_P)y_P=0, \quad i=1,2.
\end{array}
\end{equation*}
Hence  $x_P$ clearly satisfies 
\begin{equation}\label{p1}
\begin{aligned}
\varphi(x_P)=A_2(x_P)B_1(x_P)-B_2(x_P)A_1(x_P)=0.
\end{aligned}
\end{equation}
Since $\{x(P)\mid 3P=Q\}$  is well known to be the zero-set of 
\begin{equation}\label{p2}
\begin{aligned}
\phi_{3,Q}(x_P)=(x_P-x_Q)\psi_3(x_P)^2-\psi_{2}(x_P)\psi_{4}(x_P),
\end{aligned}
\end{equation}
 we deduce $x - x_P=\gcd(\varphi(x),\phi_{3,Q}(x))$ by Corollary \ref{uniquemodl}.
\end{example}

\begin{example}
Let  $E: y^2=x^3 - 9504x + 357264$ over $\mathbb{Q}$.  The point $P=(57,-27)$ has infinite order and
\[
Q=3P=\left( \frac{241977}{67^2}, \frac{11976741}{67^3} \right).
\] 
The $3$-division polynomial factors as  

\begin{align*}
\psi_3(x) &= (x-48)(x^3 + 48x^2 - 16704x + 627264)\\
&= (x-48)(x-\xi)(x^2 + (\xi + 48)x + \xi^2 + 48\xi - 16704)\\
&=(x-48)(x-\xi)(x-\rho)(x+48+\xi+\rho)
\end{align*}
with $\xi^3+48\xi^2-16704\xi +627264=0$ and $\rho^2 + (\xi + 48)\rho + \xi^2 + 48\xi - 16704=0$.

Our $E$ has full $3$-torsion over $k=\mathbb{Q}(\xi,\zeta)$  and a basis of $[3]E(k)$ is 
\[
\left\{W_1=(-48,108),\ \! W_2=\left(\xi, \frac{1}{216}\left(-\xi^2 + 96\xi - 2304\right)\rho + \frac{1}{9}\left(2\xi^2 - 300\xi + 10764\right)\right)\right\}.
\]

The Miller functions $h_1(x,y),h_2(x,y)$  are 
\begin{align*}
h_1(x,y) &= y -12x  + 684, \\
 h_2(x,y) &= y + \frac{(\xi^2 - 150\xi + 3168)\rho - 3( 25\xi^2 - 2112\xi + 79200)}{19656} x  \\
&\quad+ \frac{(107\xi^2 - 11136\xi + 417600 ) \rho - 24 ( 232\xi^2 - 34800\xi + 980676)}{19656} ,
\end{align*}
and the function $g_P$ is
\[
g_P(x,y)=\alpha_1(x_P,y_P) \left(x - \frac{21441}{67^2}\right) + \left(x - \frac{241977}{67^2} \right)^2 - \dfrac{8y_P^3}{\psi_3(x_P)} \left(y + \frac{11976741}{67^3} \right), 
\]
with  
\begin{align*}
\alpha_1(x_P,y_P) &= \dfrac{1}{67^2\cdot  \psi_3(x_P)}\big( -26934 x_P^5 + 1451862 x_P^4 + 511961472 x_P^3 + 53868 x_P^2 y_P^2  \\
&\qquad\qquad\qquad\qquad - 66087187200 x_P^2 + 2951826261120 x_P \\
&\qquad\qquad\qquad\qquad - 170653824 y_P^2 - 43713636747264 \big)
\end{align*}
After replacing $y_P^2=x_P^3 - 9504x_P + 357264$ in the coefficients, evaluating $g_P(x,y)$ at $W_1$ and $W_2$ above, and working out the $A_i(x_P), B_i(x_P)$ terms \eqref{AsBs}, then our  $\varphi(x_P)$ satisfies 
{\tiny
\begin{align*}
\varphi(x_P) &= \Bigg( \frac{(- 1271\xi^2 + 437286\xi - 23525280)\rho + (117087\xi^2 - 30772584\xi + 1504296576)}{16806816} \zeta_1 \zeta_2 \\
&\qquad + \frac{(- 394630\xi^2 +  54370590\xi - 1944830880)\rho + (7889655\xi^2 - 4989513240\xi + 286006549320)}{3941198352} \zeta_2 \\
&\qquad + \frac{(- 139159\xi + 7501287)}{601526} \zeta_1 \Bigg) x_P^9 + \ldots
\end{align*}
}
On the other hand $\phi_{3,Q}(x_P) $ satisfies 
\begin{align*}
\phi_{3,Q}(x_P) &= x_P^9 - \frac{2177793}{2289} x_P^8 + 114048 x_P^7 - \frac{71169798528}{4489} x_P^6 + \frac{5939828267904}{4489} x_P^5\\
&\quad - \frac{289893744543744}{4489} x_P^4 + \frac{7084946061840384}{4489} x_P^3 + \frac{13114237459820544}{4489} x_P^2\\
&\quad - \frac{5733182390243033088}{4489} x_P + \frac{112468959612099624960}{4489},
\end{align*}
and we have $\gcd(\varphi(x),\phi_{3,Q}(x))=x-57$. 
In terms of $\zeta_1$ and $\zeta_2$, this root can be expressed as:
{\scriptsize
\begin{align*}
\zeta_1 \zeta_2 &+ \frac{(273290\xi^2 + 32402550\xi - 2350353294)\rho + (2590587\xi^2 + 592922025\xi - 25290323496)}{53206177752} \zeta_1 \\
&+ \frac{(-5146\xi^2 - 632958\xi + 38486934)\rho + (-316479\xi^2 - 38926917\xi + 1369531152)}{1689085008} \zeta_2 \\
&+ \frac{(-20653\xi^2 - 2388804\xi + 192765960)\rho + (506934\xi^2 + 28067904\xi - 7084501056)}{6650772219}
\end{align*}
}
\end{example}

\section{Application to elliptic curve point counting}\label{sec:SEA}

In this final section, we illustrate the possible application of our approach to elliptiic curve point counting.
To do so, we compare a Magma implementation of our approach and show that at least in some curves 
it could help speed up the computation of the group order compared to Magma's own implementation.
The computations were performed on a $3.7$ GHz Quad-Core Intel Xeon E5 running macOS.

Note that a fully optimized general implementation of the algorithm combined with other point counting techniques 
is beyond the scope of this work.

\begin{example}\label{ex:SEA}
Let us consider the curve defined over $\mathbb{F}_p$, $p = 2^{703}-13867$ by
\[
y^2 = x^3 + a x + b
\]
with
\begin{align*}
a &= 30549244486266764201128458368895293006006289570265849680844392030059066\\
&\quad\ 19739647931080578316025278254574981804290063802110239743681753853367471\\
&\quad\ 7738965302993196833289497003709811897795383258861421985198534095297690 \\
b&= 10284941840037648922458883453866097590308878646908244027812197667716810\\
&\quad\ 44998822486343355981853598676086676347431098924369549279944915836614321\\
&\quad\ 0309333014640910536808722230972640809625110199256759039885067963181140
\end{align*}

When computing the group order with Magma SEA point counting implementation (in verbose mode) we see that:
\begin{itemize}
\item $2$-torsion information is used ($< 0.01$ s), but $4$-torsion information is not used;
\item $3$-torsion information is used ($< 0.01$ s), but $9$-torsion information is not used;
\item $5$-torsion information is not used (Atkin prime, $< 0.01$ s);
\item $7$-torsion information is not used (Atkin prime, $< 0.01$ s);
\item $11$-torsion information is used (Atkin prime, $0.09$ s);
\item $\ldots$
\item $313$-torsion information is used (Elkies prime, $8.58$ s);
\item $317$-torsion information is not used (Atkin prime, $3.06$ s);
\item $331$-torsion information is not used (Atkin prime, $3.12$ s);
\item $337$-torsion information is not used (Atkin prime, $3.16$ s);
\item $347$-torsion information is not used (Atkin prime, $3.16$ s);
\item $349$-torsion information is used (Elkies prime, $9.16$ s).
\end{itemize}

Note that for Elkies primes $\ell$, the values of $a_p \bmod{\ell}$ are obtained exactly, 
whereas for Atkin primes a restriction on the possible values of $a_p \bmod{\ell}$ is obtained 
and used in a Baby Step - Giant Step algorithm only if the proportion of possible values is low enough.

In our example, the Baby Step - Giant Step algorithm took $6.44$ s and the total time of the 
point counting algorithm was $210.23$ s.

With our $\ell$-section algorithm, an alternative computation becomes possible.
Rather than doing the work for the primes $317$ to $349$, we can compute points of order $\ell^k$
for small $\ell$ and use those points to extract the value of $a_p \bmod{\ell^k}$ 
(progressively increasing $k$ and the extension degree of the fields where the torsion points are defined).
Our approach took:
\begin{itemize}
\item $0.22$ s for the $2$-torsion point;
\item $0.72$ s for the $4$- and $8$-torsion points (both in the same extension);
\item $0.48$ s for the $16$-torsion point;
\item $1.74$ s for the $32$-torsion point;
\item $3.01$ s for the $64$-torsion point;
\item $0.40$ s for the $3$-torsion point;
\item $1.94$ s for the $9$-torsion point;
\item $4.89$ s for the $27$-torsion point;
\item $0.43$ s for the $5$-torsion point;
\item $3.48$ s for the $25$-torsion point.
\end{itemize}
The total time for this is $17.31$ s, compared to $21.66$ s spent by the SEA algorithm on primes $317$ to $349$.

Note that the known modulus of the precise information on $a_p$ increases by a factor of $7200/349 \approx 20.63$ 
(after ``forgetting"  $349$), which would reduce the Baby Step - Giant Step work by a factor of close to $4.54$, 
so we would expect the running time to decrease by close to $5$ s approximately.

Although the work mod $\ell$ is significantly slower than in Magma SEA implementation
(this is in part due to different way to work mod $\ell$, but could also be due to the level of optimization of the code), 
we would expect to save close to $5\%$ of the running time of the point counting algorithm.
\end{example}

\section{Conclusion}
We formulate the problem of finding $\{P\in E(k)\mid \ell P = Q\in E(k)\}$ in terms of two Miller functions $h_{W_1},h_{W_2}$ associated to a basis $\{W_1,W_2\}$ of $E[\ell]$, and a third function $g_P$ that we compute explicitly. If $P\in E(k)$ exists, then we show $G_i(P)=g_P(W_i)/h_{W_i}(P)$ are $\ell$-th roots of $h_{W_i}(-Q)$ and belong to $k$. The properties of $G_i(P)$ are similar to Weil's pairing. With them we prove that every $P$ is in bijection with a pair  of $\ell$-th roots $\omega_1,\omega_2\in k$ of $h_{W_1}(-Q)$ and $h_{W_2}(-Q)$ respectively. 
We show each pair $\omega_1,\omega_2$ determines the coordinate $x_P$ of one $P$ effectively by solving two bivariate equations. Thus we find all $P$'s at a cost of computing two $\ell$-th roots of $h_{W_i}(-Q)$, plus solving a linear equation to find $g_p$, plus the computation of a polynomial gcd.

\paragraph{Funding statement}
J.M. M. and J. P. are supported partially by grants PID2021-124613OB-I00 (Spanish Ministerio de Ciencia e Innovaci\'on) and 2021SGR 00434 (Generalitat de Catalunya). This work was partially supported by Programa Iberoamericano de Ciencia y Tecnologia (CyTeD) through red 522RT0131.


\begin{thebibliography}{00}


\bibitem{Hu} D. Husem\"{o}ller, \textit{Elliptic Curves}, 2nd ed., Grad. Texts in Math., vol. 111,  Springer-Verlag, New York, 2004. MR2024529

\bibitem{Mc} K. McMurdy, \textit{Explicit representation of the endomorphism rings of supersingular elliptic curves}, https://phobos.ramapo.edu/~kmcmurdy/research/McMurdy-ssEndoRings.pdf, preprint, 2014.

\bibitem{M} V. Miller, \textit{The Weil Pairing, and its efficient calculation}, J. Cryptology {\bf{17}} (2004), no. 4, 235--261. MR2090556 


\bibitem{MMRV} J. Miret, R. Moreno, A. Rio, M. Valls, \textit{Computing the $\ell$-power torsion of an elliptic curve over a finite field}, Math. Comp. \textbf{78} (2009), no. 267, 1767-1786.

\bibitem{PT} J. Pujol\`as and N. Th\'eriault, \textit{Trisection and cubes in elliptic curves}, J. Pure Appl. Algebra {\bf{227}} (2023), no. 4, Paper No. 107259, 11 pp. MR4513112  DOI https://doi.org/10.1016/j.jpaa.2022.107259.

\bibitem{Se} J.-P. Serre, \textit{Algebraic Groups and Class Fields}, Grad. Texts in Math., vol. 117, Springer-Verlag, New York, 1988. MR0918564

\bibitem{S} J. H. Silverman, \textit{The Arithmetic of Elliptic Curves}, Grad. Texts in Math., vol. 106, Springer-Verlag, New York, 1986. MR0817210

\bibitem{ST} J. H. Silverman and J. Tate, \textit{Rational Points on Elliptic Curves}, Undergrad. Texts  Math., Springer-Verlag, New York, 1992. MR1171452


\end{thebibliography}
\end{document}